\documentclass[oneside, 10pt]{amsart}
\usepackage{latexsym, bbm, enumerate, amssymb, amsmath, bm}
\usepackage[full,dcucite]{harvard} 
\usepackage[dvips]{graphicx}
\usepackage{setspace}
\usepackage{paralist}
\usepackage{ifpdf}

\usepackage[show]{ed}

\usepackage{tikz}

\theoremstyle{plain}
\newtheorem{theorem}{Theorem}

\newtheorem{lemma}[theorem]{Lemma}
\newtheorem{proposition}[theorem]{Proposition}
\theoremstyle{definition}

\newtheorem{remark}[theorem]{Remark}

\newcommand{\btau}{\boldsymbol{\tau}}

\newcommand{\A}{\mathcal{A}}

\newcommand{\R}{\mathbb{R}}
\newcommand{\N}{\mathbb{N}}
\newcommand{\E}{{\mathbb{E}}}

\newcommand{\F}{{\mathcal{F}}}

\DeclareMathOperator{\locmax}{locmax}

\let\oldmarginpar\marginpar
\renewcommand\marginpar[1]{\-\oldmarginpar[\raggedleft\footnotesize #1]%
{\raggedright\footnotesize #1}}

\definecolor{lightyellow}{RGB}{255,255,102}

\begin{document}

\title[]%
{Sequential Selection of a Monotone Subsequence from a Random Permutation}
\author[]
{Peichao Peng
and J. Michael Steele}

\thanks{Peichao Peng:  Department of Statistics, The Wharton School,
University of Pennsylvania, 3730 Walnut Street, Philadelphia, PA, 19104.
Email address: \texttt{ppeichao@wharton.upenn.edu}}

\thanks{J. M.
Steele:  Department of Statistics, The Wharton School,
University of Pennsylvania, 3730 Walnut Street, Philadelphia, PA, 19104.
Email address: \texttt{steele@wharton.upenn.edu}}

\citationmode{full}

\begin{abstract}
        We find a two term asymptotic expansion for the optimal expected value of a
        sequentially selected monotone subsequence from a random permutation of length $n$.
        A striking feature of this expansion is that tells us that the expected value of optimal
        selection from a random permutation
        is quantifiably larger than optimal sequential selection from an independent sequences
        of uniformly distributed random variables; specifically, it
        is larger by at least $(1/6) \log n +O(1)$.

        \bigskip


        \smallskip


\end{abstract}


\maketitle



\section{Sequential Subsequence Problems}

In the classical monotone subsequence problem, one chooses
a random permutation $\pi:[1:n] \rightarrow [1:n]$,
and one considers
the length of its longest increasing subsequence,
\begin{equation*}\label{eq:def-L-sub-n}
L_n=\max\{k: \pi[i_1] < \pi[i_2]< \cdots < \pi[i_k] \quad \text{where } 1\leq i_1< i_2 \cdots < i_k \leq n\}.
\end{equation*}
On the other hand, in the \emph{sequential} monotone subsequence problem one views the values $\pi[1]$, $\pi[2]$, ...
as though they were presented over time
to a decision maker who, when shown the value $\pi[i]$ at time $i$, must decide
(once and for all) either to accept or reject $\pi[i]$ as element of the selected increasing subsequence.

The decision to accept or reject $\pi[i]$ at time $i$ is based on just the knowledge
of the time horizon $n$ and the observed values $\pi[1], \pi[2], \ldots, \pi[i]$. Thus,
in slightly more formal language, the sequential selection problems amounts to the
consideration of random variables of the form
\begin{equation}\label{eq:def-L-tau}
L^{\btau}_n=\max\{k: \pi[\tau_1] < \pi[\tau_2]< \cdots < \pi[\tau_k] \quad \text{where }
1\leq \tau_1< \tau_2 \cdots < \tau_k \leq n \},
\end{equation}
where the indices $\tau_i$, $i=1, 2, \ldots$ are stopping times with respect to the increasing sequence of $\sigma$-fields
$\F_k=\sigma\{\pi[1], \pi[2], \ldots, \pi[k] \, \}$, $1 \leq k \leq n$.
We call a sequence of such stopping times a \emph{feasible selection
strategy}, and, if we use
${\btau}$ as a shorthand for such a
strategy, then the quantity of central interest here can be written as
\begin{equation}\label{eq:def-sn}
s(n) = \sup_{\btau} \E [L^{\btau}_n],
\end{equation}
where one takes the supremum over all feasible selection strategies.

It was conjectured in \citeasnoun{BaerBrock:MathComp68} that
\begin{equation}\label{eq:s(n)-asymptotics}
s(n) \sim \sqrt{2n} \quad \text{as } n \rightarrow \infty,
\end{equation}
and a proof of this relation was first given
in \citeasnoun{SamSte:AP1981}. A much simpler proof of \eqref{eq:s(n)-asymptotics}
was later given by \citeasnoun{Gne:CPC2000} who made use of a recursion that had been used for numerical computations by
\citeasnoun{BaerBrock:MathComp68}.
The main purpose of this note is to show that by a more sustained investigation of that
recursion one can obtain a two  term expansion.

\begin{theorem}[Sequential Selection from a Random Permutation]\label{thm:seq-select-thm} For $n \rightarrow \infty$ one has
the asymptotic relation
\begin{equation}\label{eq:s(n)-asymp}
s(n) = \sqrt{2n} + \frac{1}{6} \log n + O(1).
\end{equation}
\end{theorem}

Given what is known for some closely related problems, the explicit second order term $(\log n)/6$ presents something of a surprise.
For comparison, suppose we consider sequential selection from a sequence of $n$ independently uniformly distributed
random variables $X_1$, $X_2$, ..., $X_n$. In this problem a feasible selection strategy $\btau$ is again expressed by
an increasing sequence of stopping times $\tau_j$, $j=1,2, \ldots$, but now the stopping times are adapted to the increasing $\sigma$-fields
$\widehat{\F}_j=\sigma\{X_1, X_2, \ldots, X_j\}$. The analog of \eqref{eq:def-L-tau} is then
\begin{equation}\label{eq:def-L-hat-tau}
\widehat{L}^{\btau}_n=\max\{k: X_{\tau_1} < X_{\tau_2}< \cdots < X_{\tau_k} \quad \text{where }
1\leq \tau_1< \tau_2 \cdots < \tau_k \leq n \},
\end{equation}
and the analog of \eqref{eq:def-sn} is given by
\begin{equation*}\label{eq:def-ell-n}
\widehat{s}(n) = \sup_{\btau} \E [\widehat{L}^{\btau}_n].
\end{equation*}

It was proved by
\citeasnoun{BruRob:AAP1991} that for $\widehat{s}(n)$ one has a uniform upper bound
\begin{equation}\label{eq:BRUpperBound}
\widehat{s}(n) \leq \sqrt{2n} \quad \text{for all } n \geq 1,
\end{equation}
so, by comparison with  \eqref{eq:s(n)-asymp}, we see there is a sense in which sequential selection
of a monotone subsequence from a permutation is \emph{easier} than sequential selection from an independent sequence.
In part, this is intuitive;
each successive observation from a permutation gives useful information about the subsequent values that can be observed.
By \eqref{eq:s(n)-asymp} one quantifies how much this information helps.

Since \eqref{eq:BRUpperBound} holds for all $n$ and since \eqref{eq:s(n)-asymp} is only asymptotic, it also seems natural to ask if there is a
relation between $\widehat{s}(n)$ and  $s(n)$ that is valid for all $n$. There is such a relation if one gives up the logarithmic gap.

\begin{theorem}[Selection for Random Permutations vs Random Sequences]\label{thm:iid-v-perm}
One has for all $n=1,2,\ldots$ that
\begin{equation*}\label{eq:iid-v-perm}
\widehat{s}(n) \leq s(n).
\end{equation*}
\end{theorem}

Here we should also note that  much more is
known about $\widehat{s}(n)$ than just \eqref{eq:BRUpperBound}; in particular, there are several further connections between
$s(n)$ and $\widehat{s}(n)$.
These are taken up in a later section, but first it will be useful to give the proofs
of Theorems \ref{thm:seq-select-thm} and \ref{thm:iid-v-perm}.

The proof of Theorem \ref{thm:seq-select-thm} takes most of our effort, and it is given over the next few sections.
Section \ref{se:recurrence} develops the basic
recurrence relations, and Section \ref{se:Stable Comparison Relations} develops stability relations for these
recursions. In Section \ref{se:Approximations} we then do the calculations that support a candidate for the asymptotic
approximation of $s(n)$, and we compete the proof of Theorem \ref{thm:seq-select-thm}. Our arguments conclude
in Section \ref{sec:ProofThm2} with the brief --- and almost computation free --- proof of Theorem \ref{thm:iid-v-perm}.
Finally, in Section \ref{se:FurtherConnections} we discuss further relations between $s(n)$, $\widehat{s}(n)$, and some
other closely related
quantities that motivate two open problems.

\section{Recurrence Relations}\label{se:recurrence}

One can get a recurrence relation for $s(n)$ by first step analysis.
Specifically, we take a random permutation $\pi:[1:n+1] \rightarrow [1:n+1]$, and
we consider its initial value $\pi[1]=k$.
If we reject  $\pi[1]$ as an element of our subsequence, we are faced with the problem of sequential
selection from the reduced random permutation $\pi'$ on an $n$-element set.
Alternatively, if we choose $\pi[1]=k$ as an element of our subsequence, we are then faced
with the problem of sequential selection for a reduced random permutation $\pi''$ of the set $\{k+1,k+2, \ldots, n+1\}$ that
has $n+1-k$ elements.
By taking the better of these
two possibilities, we get from the uniform distribution of $\pi[1]$ that
\begin{align}\label{eq:inside recursion}
s(n+1) = \frac{1}{n+1} \sum_{k=1}^{n+1}\max \{ s(n) ,1+s(n+1-k) \}.
\end{align}
From the definition \eqref{eq:def-sn} of $s(n)$ one has $s(1)=1$, so subsequent values can then be computed by \eqref{eq:inside recursion}.
For illustration and for later discussion, we note that one has
the approximate values:

\medskip
\begin{tabular}{ r | c c c c c c c c c c }			
 $n$ & 1 & 2 & 3 & 4 & 5 & 6 & 7 & 8 & 9 & 10 \\
 \hline
 $s(n)$ & $1$ & $1.5$ & $2$ & $2.375$ & $2.725$ & $3.046$ & $3.333$ &  $3.601$ & $3.857$ & $4.098$ \\
$\sqrt{2n}$ & $1.414$ & $2$ & $2.449$ & $2.828$ & $3.162$ & $3.464$ & $3.742$ & $4$ & $4.243$ & $4.472$. \\
\end{tabular}
\smallskip

Here we observe that for the $10$ values in the table one has $s(n) \leq \sqrt{2n}$, and, in fact, this relation persists
for all $1\leq n \leq 174$. Nevertheless, for $n=175$ one has $\sqrt{2n} < s(n)$, just as \eqref{eq:s(n)-asymp} requires for all
sufficiently large values of $n$.

\medskip

We also know from \eqref{eq:def-sn} that the map $n \mapsto s(n)$ is strictly monotone increasing, and, as a consequence, the
recursion \eqref{eq:inside recursion} can be written a bit more simply as
\begin{align}\label{eq:outside recursion}
s(n+1) &= \frac{1}{n+1} \max_{1\le k\le n}\left\{(n-k+1)s(n) + \sum_{i=n-k+1}^n \{s(i)+1\}\right\} \\
&= \frac{1}{n+1} \max_{1\le k\le n}\left\{(n-k+1)s(n) + k+ \sum_{i=n-k+1}^n s(i)\right\}. \notag
\end{align}
In essence, this recursion goes back to \citeasnoun[p.~408]{BaerBrock:MathComp68}, and it is the basis of most of our analysis.

\section{Comparison Principles}\label{se:Stable Comparison Relations}

Given a map $g: \N \rightarrow \R$ and  $1 \leq k \leq n$, it will be convenient to set
\begin{equation}\label{def:H}
H(n,k,g) = k+(n-k+1)g(n) + \sum_{i=n-k+1}^n g(i),
\end{equation}
so the optimality recursion \eqref{eq:outside recursion} can be written more succinctly as
\begin{equation}\label{eq:HandS}
s(n+1)=    \frac{1}{n+1} \max_{1\le k\le n} H(n,k, s).
\end{equation}
The next two lemmas make rigorous the idea that if $g$ is almost a solution of \eqref{eq:HandS} for all $n$, then
$g(n)$ is close to $s(n)$ for all $n$.

\begin{lemma}[Upper Comparison]\label{lm:s-bound-by-deltas}  If $\delta: \N \rightarrow \R^+$, $1 \leq g(1) + \delta(1)$, and
\begin{equation}\label{eq:InductionCondition}
\frac{1}{n+1} \max_{1\leq k \leq n} H(n,k,g) \leq g(n+1) + \delta(n+1) \quad \text{for all } n\geq 1,
\end{equation}
then one has
\begin{equation}\label{eq:s-bound-by-deltas}
s(n) \leq g(n) + \sum_{i=1}^n \delta (i) \quad \text{for all } n\geq 1.
\end{equation}
\end{lemma}
\begin{proof}
We set $\Delta(i)=\delta(1)+\delta(2) +\cdots + \delta(i)$, and we argue by induction. Specifically,
using \eqref{eq:s-bound-by-deltas} for $1 \leq i \leq n$ we have
\begin{align*}
H(n,k,s) &= k +(n-k+1)s(n) + \sum_{i=n-k+1}^n s(i)\\
&\leq k +(n-k+1)(g(n)+\Delta(n))  + \sum_{i=n-k+1}^n \{g(i)+\Delta(i)\}
\end{align*}
so by monotonicity of $\Delta(\cdot)$ we have
$$
\frac{1}{n+1} H(n,k,s) \leq \frac{1}{n+1} H(n,k,g)+ \Delta(n).
$$
Now, when we take the maximum over $k \in [1:n]$,  the recursion \eqref{eq:outside recursion} and the
induction condition \eqref{eq:InductionCondition},
give us
\begin{align*}
s(n+1) &\leq \frac{1}{n+1} \max_{1\leq k \leq n} H(n,k,g)+\Delta(n)\\
&\leq g(n+1) +\delta(n+1) + \Delta(n)=g(n+1)+\Delta(n+1),
\end{align*}
so induction establishes \eqref{eq:s-bound-by-deltas} for all $n\geq 1$.
\end{proof}

Naturally, there is a lower bound comparison principle that parallels Lemma \ref{lm:s-bound-by-deltas}. The statement has
several moving parts, so we frame it as a separate lemma even though its proof can be safely omitted.

\begin{lemma}[Lower Comparison]\label{lm:s-bound-by-deltas-below} If $\delta: \N \rightarrow \R^+$, $g(1) - \delta(1) \leq  1 $, and
\begin{equation*}\label{eq:InductionCondition-lower}
g(n+1) - \delta(n+1)  \leq \frac{1}{n+1} \max_{1\leq k \leq n} H(n,k,g) \quad \text{for all } n\geq 0,
\end{equation*}
then one has
\begin{equation*}\label{eq:s-bound-by-deltas-below}
g(n) - \sum_{i=1}^n \delta (i)  \leq s(n) \quad \text{for all } n\geq 1.
\end{equation*}
\end{lemma}

\section{An Approximation Solution}\label{se:Approximations}

We now argue that the function $f: \N \rightarrow \R$ defined by
\begin{equation}\label{eq:f-def}
f(n) =\sqrt{2n} + \frac{1}{6} \log n,
\end{equation}
gives one an approximate solution of the recurrence equation \eqref{eq:outside recursion} for $n \mapsto s(n)$.

\begin{proposition}\label{pr:f-prop} There is a constant $0< B < \infty$ such that for all $n\geq 1$, one has
\begin{equation}\label{eq:f-bound-on-H}
-Bn^{-3/2}\leq \frac{1}{n+1} \left\{\max_{1\leq k \leq n} H(n,k,f)\right\} -f(n+1) \leq B n^{-3/2}.
\end{equation}
\end{proposition}

\medskip
\noindent
{\sc First Step: Localization of the Maximum}
\smallskip

To deal with the maximum in \eqref{eq:f-bound-on-H}, we first estimate
\begin{equation*}\label{eq:defK}
k^*(n) = \locmax_k \, H(n,k, f).
\end{equation*}
From the definition \eqref{def:H} of $H(n,k,f)$ we find
$$
H(n,k+1,f)-H(n,k,f)=1-f(n)+ f(n-k),
$$
and, from the definition \eqref{eq:f-def} of $f$, we see this difference is monotone decreasing function of $k$; accordingly, we also have
the representation
\begin{equation}\label{eq:char-k}
k^*(n) =1+\max\{ k: 0 \leq 1-f(n)+ f(n-k) \}.
\end{equation}
Now, for each $n=1,2, \ldots$ we  then consider the function $D_n: [0,n] \rightarrow \R$ defined by
setting
$$
D_n(x)= 1-f(n)+ f(n-x)= 1-\{\sqrt{2n}- \sqrt{2(n-x)} \}-\frac{1}{6} \{\log n  -\log (n-x)\}.
$$
This function is strictly decreasing with $D_n(0)=1$ and $D_n(n)=-\infty$, so there is a unique solution
of the equation $D_n(x)=0$.  For $x \in [0, n)$ we also have the easy bound
$$
D_n(x) = 1 -\frac{1}{2} \int_{2(n-x)}^{2n} \frac{1}{\sqrt{u}} \, du -\frac{1}{6} \log(n/(n-x)) \leq 1 - \frac{x}{\sqrt{2n}}.
$$
This gives us $D_n(\sqrt{2n})\leq 0$, so by monotonicity we have $x_n \leq \sqrt{2n}$.

To refine this bound to an asymptotic estimate, we start with he equation $D_n(x_n)=0$ and we apply Taylor expansions to get
\begin{align*}
1&= \sqrt{2n} \left\{ 1- (1-x_n/n)^{1/2}\right\} -\frac{1}{6} \log (1- x_n/n) \\
&=\sqrt{2n} \left\{ \frac{x_n}{2n} +O(x_n^2/n^2) \right\} + O(x_n/n).
\end{align*}
By simplification,  we then get
\begin{equation}\label{eq:xnBigO}
\sqrt{2n}=x_n+O(x_n^2/n) + O(x_n/n^{1/2}) = x_n +O(1),
\end{equation}
where in the last step we used our first bound $x_n \leq \sqrt{2n}$.

Finally, by \eqref{eq:xnBigO} and the characterization \eqref{eq:char-k},  we immediately find the estimate that
we need for $k^*(n)$.
\begin{lemma}\label{lm:f-breakpoint} There is a constant $A>0$ such that for
all $n\geq 1$, we have
\begin{equation}\label{eq:Kest}
\sqrt{2n} -A\leq k^*(n) \leq \sqrt{2n} +A.
\end{equation}
\end{lemma}

\medskip

\begin{remark} The relations \eqref{eq:xnBigO} and \eqref{eq:Kest}
can be sharpened. Specifically, if we use a
two-term Taylor series with integral remainders, then one can show
$\sqrt{2n} -2 \leq x_n$. Since we already know that $x_n \leq \sqrt{2n}$, we then see from
the characterization \eqref{eq:char-k} and integrality of $k^*(n)$
that we can take $A=2$ in Lemma \ref{lm:f-breakpoint}. This refinement does not lead to a meaningful improvement
in Theorem \ref{thm:seq-select-thm}, so we omit the details of the expansions with remainders.
\end{remark}

\medskip
\noindent
{\sc Completion of Proof of Proposition  \ref{pr:f-prop} }
\medskip

To prove Proposition \ref{pr:f-prop}, we first note that the definition \eqref{def:H} of
$H(n,k,f)$ one has for all $1\leq k \leq n$ that
\begin{equation}\label{eq:H-sum-f(n)}
\frac{1}{n+1} H(n,k,f)= f(n) + \frac{1}{n+1}\left\{ k - \sum_{i=1}^{k-1} (f(n)-f(n-i)) \right\}
\end{equation}
The task is to estimate the right-hand side of \eqref{eq:H-sum-f(n)} when $k=k^*(n)$  and  $k^*(n)$ is given by \eqref{eq:char-k}.

For the moment, we assume that one has $k \leq D \sqrt n$  where $D>0$ is constant.
With this assumption, we find that after making
Taylor expansions we get from explicit summations that
\begin{align}
\sum_{i=1}^{k-1}& (f(n)-f(n-i))  = \sum_{i=1}^{k-1} \left(\sqrt{2n} - \sqrt{2(n-i)}\right)
+ \sum_{i=1}^{k-1} \left(\frac{\log n}{6} - \frac{\log(n-i)}{6}\right) \notag\\
& = \sum_{i=1}^{k-1}\left(\frac{i}{\sqrt{2n}}+\frac{i^2}{4n\sqrt{2n}}+O\left(\frac{i^3}{n^{5/2}}\right)\right) + \sum_{i=1}^{k-1}\left(\frac{i}{6n}+O\left(\frac{i^2}{n^2}\right)\right) \notag\\
& = \frac{(k-1)k}{2\sqrt{2n}}+\frac{(k-1)k(2k-1)}{24n\sqrt{2n}} + \frac{(k-1)k}{12n} + O(n^{-1/2}),\label{eq:BigSum}
\end{align}
where the implied constant of the remainder term depends only on $D$.

We now define $r(n)$ by the relation $k^*(n)=\sqrt{2n} +r(n)$, and we note by \eqref{eq:Kest} that $|r(n)| \leq A$.
Direct algebraic expansions then give us the elementary estimates
$$
\frac{(k^*(n)-1)k^*(n)}{12n}=\frac{1}{6} + O(n^{-1/2})
$$
and
$$
\frac{(k^*(n)-1)k^*(n)(2k^*(n)-1)}{24n\sqrt{2n}}=\frac{1}{6} + O(n^{-1/2}),
$$
where in each case the implied constant depends only on $A$.

Estimation of the first summand of \eqref{eq:BigSum} is
slightly more delicate than this since we need to account for the dependence of this term on $r(n)$; specifically we have
\begin{align*}
\frac{(k^*(n)-1)k^*(n)}{2\sqrt{2n}}& =\frac{\left(\sqrt{2n}+r(n)-1\right)\left(\sqrt{2n}+r(n)\right)}{2\sqrt{2n}}\\
&=\sqrt{n/2} +r(n) - \frac{1}{2}+ O(n^{-1/2}).
\end{align*}
Now, for a pleasing surprise, we note from the last estimate and from the definition of $k^*(n)$ and $r(n)$ that we have
cancelation of $r(n)$  when we then compute the critical sum; thus, one has simply
\begin{equation}\label{eq:kminusSum}
k^*(n) - \sum_{i=1}^{k^*(n)-1} (f(n)-f(n-i)) = \sqrt{n/2} +\frac{1}{6} +O(n^{-1/2}).
\end{equation}

Finally, from the formula \eqref{eq:f-def} for $f(\cdot)$, we have the Taylor expansion
\begin{equation}\label{eq:delta-f-expansion}
f(n+1) -f(n)= \frac{1}{\sqrt{2n}} + \frac{1}{6n} + O(n^{-3/2}),
\end{equation}
so, when we return to the identity
\eqref{eq:H-sum-f(n)},  we see that the estimates \eqref{eq:kminusSum}
and \eqref{eq:delta-f-expansion} give us the estimate
\begin{align*}
\frac{1}{n+1} &\left\{\max_{1\leq k \leq n} H(n,k,f)\right\} -f(n+1) \\
&=
\frac{1}{n+1} \left( \sqrt{n/2} +\frac{1}{6} +O(n^{-1/2}) \right) +f(n) -f(n+1) =O(n^{-3/2}).
\end{align*}
Here the implied constant is absolute, and the proof of Proposition \ref{pr:f-prop} is complete.

\medskip
\noindent
{\sc Completion of Proof of Theorem  \ref{thm:seq-select-thm} }
\medskip

Lemmas \ref{lm:s-bound-by-deltas} and \ref{lm:s-bound-by-deltas-below} combine with  Proposition \ref{pr:f-prop}
to tell us that by summing the sequence $n^{-3/2}$,  $n=1,2,\ldots$ and by writing $\zeta(z)=1+2^{-z}+3^{-z}+\cdots$ one has
\begin{equation*}\label{eq:first-approx-bound}
|s(n) -f(n)| \leq \zeta(3/2) B \leq (2.62) B \quad \text{for all } n \geq 1.
\end{equation*}
This is slightly more than one needs to complete the proof of Theorem \ref{thm:seq-select-thm}.

\section{Proof of Theorem \ref{thm:iid-v-perm}}\label{sec:ProofThm2}

The sequential monotone selection problem is a finite horizon Markov decision problem
with bounded rewards and finite action space, and for such problems it is known
one cannot improve upon an optimal deterministic strategy by the use of strategies that incorporate randomization,
\citeaffixed[Corollary 8.5.1]{BerShr:AP1978}{cf.}. The proof Theorem \ref{thm:iid-v-perm} exploits this observation
by constructing a randomized algorithm for the sequential selection of a monotone subsequence from a random permutation.

We first recall that if $e_i$, $i=1,2,\dots, n+1$ are
independent exponentially distributed random variables with mean $1$ and if one sets
$$
Y_i=\frac{e_1+e_2 + \cdots + e_i}{e_1+e_2 + \cdots + e_{n+1}},
$$
then the vector $(Y_1,Y_2, \ldots, Y_n)$ has the same distribution as the vector of order statistics $(X_{(1)}, X_{(2)}, \ldots, X_{(n)})$
of an i.i.d.~sample of size $n$ from the uniform distribution.
Next we let $\A$ denote an optimal algorithm for sequential selection from an independent
sample $X_1, X_2, \dots, X_n$ from the uniform distribution,
and we let $\btau(\A)$ denote the associated sequence of stopping times. If
$\widehat{L}^{\btau(\A)}_n$ denotes the length of the subsequence that is chosen from
from $X_1, X_2, \dots, X_n$ when one follows the strategy $\btau(\A)$ determined by $\A$, then by optimality of $\A$ for selection from
$X_1, X_2, \dots, X_n$ we have
$$
\widehat{s}(n) = \sup_{\btau} \E [\widehat{L}^{\btau}_n]=\E[\widehat{L}^{\btau(\A)}_n].
$$

We use the algorithm $\A$ to construct a new randomized algorithm $\A'$ for sequential
selection of an increasing from a random permutation $\pi: [n]\mapsto [n]$.
First, the decision maker generates independent exponential random variables $e_i$, $i=1,2,\dots, n+1$ as above. This is done off-line, and
this step can be viewed as an internal randomization.

Now, for $i=1,2, \ldots, n$, when we are presented with $\pi[i]$ at time $i$, we compute
$X_i =Y_{\pi[i]}$. Finally, if at time $i$ the value $X_i$ would be accepted by the algorithm $\A$,
then the algorithm $\A'$ accepts $\pi[i]$. Otherwise the newly observed value $\pi[i]$ is rejected. By our construction we have
\begin{equation}\label{eq:AprimeValue}
\E[{L}^{\btau(\A')}_n]=\E[\widehat{L}^{\btau(\A)}_n]=\widehat{s}(n).
\end{equation}

Moreover, $\A'$ is a randomized algorithm for construction an increasing subsequence of a random permutation $\pi$.
By definition, $s(n)$ is the expected length of a
monotone subsequence selected from a random permutation by an optimal deterministic algorithm, and by our earlier observation,
the randomized algorithm $\A'$ cannot do better. Thus, from \eqref{eq:AprimeValue} one has $\widehat{s}(n) \leq s(n)$,
and the proof of Theorem \ref{thm:seq-select-thm} is complete.

\section{Further Connections and Considerations}\label{se:FurtherConnections}

As we noted before, \citeasnoun{BruRob:AAP1991} discovered the uniform bound
\begin{equation}\label{eq:BRagain}
\widehat{s}(n) \leq \sqrt{2n} \quad \text{for all } n \geq 1,
\end{equation}
and their proof depended on a general bound for the expected value of the partial sums of the smallest order statistics of a uniformly
distributed random sample.
\citeasnoun{Gne:JAP1999} later gave a much different proof of \eqref{eq:BRagain} and generalized the bound in a way that
accommodate random samples with random sizes.
More recently, \citeasnoun{ArlottoMosselSteele:2015} obtained yet another proof \eqref{eq:BRagain} as a corollary to
bounds on the \emph{quickest selection problem}, which is an informal dual to the traditional selection problem.

Since the bound \eqref{eq:BRagain} is now well understood from several points of view, it is reasonable to ask about the possibility
of some corresponding uniform bound on $s(n)$. The numerical values that we noted after the
recursion \eqref{eq:BRagain} and the relation
\begin{equation}\label{eq:s(n)-asymp-bd-2}
s(n) = \sqrt{2n} + \frac{1}{6} \log n + O(1)
\end{equation}
from Theorem \ref{thm:seq-select-thm} both tell us that one cannot expect a uniform bound for $s(n)$ that is as simple as that
for $\widehat{s}(n)$ given by \eqref{eq:BRagain}.
Nevertheless, numerical evidence suggest that the $O(1)$
term in \eqref{eq:s(n)-asymp-bd-2} is always negative. The tools used here cannot confirm this conjecture, but the multiple perspectives
available for  \eqref{eq:BRagain} give one hope.

A closely related issue arises for $\widehat{s}(n)$ when one considers lower bounds. Here the first steps were taken
by \citeasnoun{BruDel:SPA2001} who considered i.i.d.~samples of size $N_\nu$ where $N_\nu$ is an independent random variable with
the Poisson distribution with mean $\nu$. In the natural (but slightly overburdened) notation, they proved that there is a constant $c>0$ such that
$$
\sqrt{2 \nu} -c \log \nu \leq \widehat{s}(\nu);
$$
moreover, \citeasnoun{BruDel:SPA2004} subsequently proved that for the optimal feasible strategy $\btau_*=(\tau_1, \tau_2, \ldots)$ the random variable
\begin{equation*}\label{eq:def-L-hat-tau}
\widehat{L}^{\btau_*}_{N_\nu}=\max\{k: X_{\tau_1} < X_{\tau_2}< \cdots < X_{\tau_k} \quad \text{where }
1\leq \tau_1< \tau_2 \cdots < \tau_k \leq N_{\nu} \},
\end{equation*}
also satisfies a central limit theorem. \citeasnoun{ArlNguSte:SPA2015}
considered the de-Poissonization of these results, and it was found that
one has the corresponding CLT for $\widehat{L}^{\btau_*}_{n}$ where the sample size $n$ is deterministic. In particular, one has the bounds
$$
\sqrt{2n} - c \log n \leq \widehat{s}(n) \leq \sqrt{2n}.
$$
Now, by analogy with \eqref{eq:s(n)-asymp-bd-2}, one strongly expects that there is a constant $c>0$ such that
\begin{equation}\label{eq:BestConjecture}
\widehat{s}(n) = \sqrt{2n}-c \log n + O(1).
\end{equation}
Still, a proof this conjecture is reasonably remote, since, for the moment,
there is not even a compelling candidate for the value of $c$.

For a second point of comparison, one can recall the \emph{non-sequential} selection problem where one studies
$$
\ell(n)= \E [ \max\{k: X_{i_1}< X_{i_2}< \ldots < X_{i_k}, \, \, 1\leq i_1< i_2< \cdots <i_k \leq n \}].
$$
Through a long sequence of investigations culminating with  \citeasnoun{BaikDeiftKurt99}, it is now known that
one has
\begin{equation}\label{eq:BDK99}
\ell(n)= 2 \sqrt{n} -\alpha n^{1/6} + o(n^{1/6}),
\end{equation}
where the constant $\alpha=1.77108...$ is determined numerically in terms of solutions of a Painlev\'e equation of type II.
\citeasnoun{Rom:CUP2014} gives an elegant account of the extensive technology behind \eqref{eq:BDK99}, and there are
interesting analogies between $\ell(n)$ and $s(n)$. Nevertheless, a proof of the conjecture \eqref{eq:BestConjecture} seems much more likely to
come from direct methods like those used here to prove \eqref{eq:s(n)-asymp-bd-2}.

Finally, one should note that the asymptotic formulas for $n\mapsto \ell(n)$, $n\mapsto s(n)$,
and $n\mapsto \widehat{s}(n)$ all suggest that these maps are concave, but so far only $n\mapsto \widehat{s}(n)$
has been proved to be concave (cf.~\citeasnoun[p. 3604]{ArlNguSte:SPA2015}).

\end{document}